\documentclass[a4,12pt]{article}
\usepackage{amsmath}
\usepackage{amssymb}
\usepackage{amsthm}
\newtheorem{lemma}{Lemma}

\newtheorem{theorem}{Theorem}
\newtheorem*{kr}{Theorem A}
\theoremstyle{definition}
\newtheorem{remark}{Remark}
\newtheorem{note}{Note}
\newtheorem*{corrections}{Corrections to [13]}

\begin{document}
\title{Note on relations among multiple zeta(-star) values II}
\author{Masahiro Igarashi}
\date{}
\maketitle
\begin{abstract}
I obtain new evaluations of special values of multiple polylogarithms by using a limiting case of a basic hypergeometric identity  of G. E. Andrews.
\end{abstract}
\begin{flushleft}
\end{flushleft}
\begin{flushleft}
\end{flushleft}
\section{Introduction}
In the present paper, we deal with the following special values of multiple polylogarithms (MPLs for short): 
\begin{equation}
\sum_{0< m_1< \cdots<m_n<\infty}
\frac{z^{m_n-1}}{m_1^{k_1} \cdots m_n^{k_n}},
\end{equation}
\begin{equation}
\sum_{1{\le}m_1\le\cdots{\le}m_n<\infty}
\frac{z^{m_n-1}}{m_1^{k_1} \cdots m_n^{k_n}},
\end{equation}
where $z=\pm1$, $1{\le}k_i\in\mathbb{Z}$ $(i=1,\ldots,n-1)$ and $2{\le}k_n\in\mathbb{Z}$. 
If $z=1$, the values (1) and (2) become the multiple zeta value $\zeta(\{k_i\}^{n}_{i=1})$ (MZV for short) and the multiple zeta-star value $\zeta^{\star}(\{k_i\}^{n}_{i=1})$ (MZSV for short), respectively, where $\{k_i\}^{n}_{i=1}:=k_1,\dots,k_n$  (see Euler \cite{eu}, Hoffman \cite{h}, and Zagier \cite{z}). We denote by $\zeta_{-}(\{k_i\}^{n}_{i=1})$ and 
$\zeta^{\star}_{-}(\{k_i\}^{n}_{i=1})$ the case $z=-1$ of (1) and of (2), respectively. It is known that special values of MPLs satisfy numerous relations 
(see, e.g., \cite{bbv}, \cite{bbbl}, and \cite{bb}). In the present paper, 
we deal also with the generalized hypergeometric series
\begin{equation*}
{}_{p+1}F_p\left(
\begin{array}{c}
a_1,\ldots,a_{p+1}\\
b_1,\ldots,b_p
\end{array}
;z\right)
:=
\sum_{m=0}^{\infty}\frac{(a_1)_m\cdots(a_{p+1})_m}{(b_1)_m\cdots(b_p)_m}\frac{z^m}{m!},
\end{equation*}
where $1{\le}p\in\mathbb{Z}$, $z$, $a_1,\ldots,a_{p+1}\in\mathbb{C}$, $b_1,\ldots,b_p\in\mathbb{C}\setminus\{0,-1,-2,\ldots\}$; and $(a)_m$ denotes 
the Pochhammer symbol, i.e., $(a)_m=a(a+1)\cdots(a+m-1)$ 
($1{\le}m\in\mathbb{Z}$) and $(a)_0=1$. This power series converges absolutely for all $z\in\mathbb{C}$ such that $|z|=1$ provided $\mathrm{Re}\,(\sum_{i=1}^{p}b_i-\sum_{i=1}^{p+1}a_i)>0$. 
In \cite{i2011} and \cite{i4}, I showed that various new and interesting relations among special values of MPLs can be derived from 
the following limiting case of Andrews' basic hypergeometric identity \cite[Theorem 4]{an}: 
\begin{kr}[Krattenthaler and Rivoal {\cite[Proposition 1 (i)]{kr}}, 
a limiting case of {\cite[Theorem 4]{an}}]
Let $s$ be a positive integer and $a$, $b_i$, $c_i$ $(i=1,\ldots,s+1)$ complex numbers. 
Suppose that the complex numbers $a$, $b_i$, $c_i$ $(i=1,\ldots,s+1)$ satisfy the conditions 
$1+a-b_i,\,1+a-c_i\notin\{0,-1,-2,\ldots\}$ $(i=1,\ldots,s+1),$
\begin{equation*}
\begin{aligned}
&\mathrm{Re}\left((2s+1)(a+1)-2\sum_{i=1}^{s+1}(b_i+c_i)\right)>0,\\
&\mathrm{Re}\left(\sum_{i=r}^{s+1}A_i(1+a-b_i-c_i)\right)>0 \quad (r=2,\ldots,s+1)
\end{aligned}
\end{equation*}
for all possible choices of $A_i=1$ or $2$ $(i=2,\ldots,s)$, $A_{s+1}=1$. 
$($For details of the choices of $A_i$, see \cite{kr}.$)$ 
Then 
\begin{equation*}
\begin{aligned}
&{}_{2s+4}F_{2s+3}\left(
\begin{array}{c}
a, \frac{a}{2}+1, b_1, c_1, \ldots, b_{s+1}, c_{s+1}\\
\frac{a}{2}, 1+a-b_1, 1+a-c_1, \ldots, 1+a-b_{s+1}, 1+a-c_{s+1}
\end{array}
;-1\right)\\
=
&\frac{\Gamma(1+a-b_{s+1})\Gamma(1+a-c_{s+1})}{\Gamma(1+a)\Gamma(1+a-b_{s+1}-c_{s+1})}\\
&\times\sum_{l_1,\ldots,l_s=0}^{\infty}\prod_{i=1}^{s}
\frac{(1+a-b_i-c_i)_{l_i}(b_{i+1})_{l_1+\cdots+l_i}(c_{i+1})_{l_1+\cdots+l_i}}
{l_i!(1+a-b_i)_{l_1+\cdots+l_i}(1+a-c_i)_{l_1+\cdots+l_i}}.
\end{aligned}
\end{equation*}
\end{kr}
Andrews' hypergeometric identity \cite[Theorem 4]{an} is a terminating one. 
The justification of taking the limit $N {\rightarrow} \infty$ in the identity was rigorously done by Krattenthaler and Rivoal in \cite{kr}. 
They applied Theorem A to proving an identity between hypergeometric series and multiple integrals related to construction of $\mathbb{Q}$-linear forms in the Riemann zeta values $\zeta(k)$, which is Zudilin's identity \cite[Theorem 5]{zu}. Their study is an application of Andrews' identity to diophantine problems of zeta values. The application to the study of relations among special values of MPLs was studied in my previous 
papers \cite{i2011} and \cite{i4}. In these papers, I showed the usefulness of Andrews' identity for such a study. In fact, I proved various new and interesting relations among special values of MPLs by using Theorem A: for example, the identities (29), (30), (28), (58) of \cite{i4}; see also \cite[(\textbf{A3}), (\textbf{A4}), Addendum]{i2011}. 
I note that my calculus of the Pochhammer symbol $(a)_m$ used in \cite{i4} 
can be applied to deriving multiple series identities from other hypergeometric identities. (My calculus was developed in the preprints of \cite{i4} distributed in 2013.) In the present paper, I again deal with the application of Theorem A, and obtain new evaluations of special values of MPLs: see Theorems 1, 2, and 3 below. 
The present paper and my previous papers \cite{i2011}, \cite{i4} are motivated by \cite[Theorem 4]{an}, \cite[Proposition 1]{kr}, and \cite[Remarks 2.6 and 2.7]{i1}: 
the hypergeometric identities of Andrews, Krattenthaler and Rivoal motivated me to study hypergeometric methods of proving relations among special values of MPLs. My research on the application of their hypergeometric identities was begun in 2009. (I had already noted the paper \cite{kr} in my preprint arXiv:0908.2536v1 in 2009.) The identities (29), (30), (28), (58) of \cite{i4} 
and their hypergeometric proofs are some of my original results in my research. 
My identities (29), (30) and (28), (58) were proved in 2011 and 2014, 
respectively. 
These four identities of mine were discovered by using \cite[Proposition 1]{kr}. 
The first two identities show the rightness of my observations stated in \cite[(\textbf{A3}), (\textbf{A4})]{i2011}; see \cite[Addendum]{i2011}. 
See also Note 2 at the end of Section 2. 
\begin{note} 
I obtained Applications 1 and 2 below in April--July 2013 and June 2014, respectively. These results of mine were written in unpublished earlier versions of 
\cite{i4} (2013--2014), but not written in the published version \cite{i4}, though Application 1 was written in 
the earlier versions of \cite{i4} distributed as preprints on October 28 and December 16, 2013. Application 1 was written also in the manuscript of \cite{i4} 
submitted to the Nagoya Mathematical Journal on March 9, 2015 and 
in its revised versions submitted to other journals in 2015--2016. 
After that, I wrote both of Applications 1 and 2 in an unpublished paper, which was submitted to several journals in December 2018--June 2019. 
(In the paper, Applications 1 and 2 were written as appendices.) 
I separated Applications 1 and 2 from the paper, and compiled them into the present paper. The present paper was submitted to many journals in July 2019--2021. 
\end{note} 
\section{Applications of Andrews' hypergeometric identity}
Hereafter we use the following notation: 
\begin{equation*}
\begin{aligned}
&\mathbb{Z}_{\ge{k}}:=\{k,k+1,k+2,\ldots\}, \quad \mathbb{Z}_{\le{k}}:=\{k,k-1,k-2,\ldots\},\\
&\{a\}^{n}:=\underbrace{a,\ldots,a}_{n},
\end{aligned}
\end{equation*}
where $k\in\mathbb{Z}$ and $1{\le}n\in\mathbb{Z}$; we regard $\{a\}^{0}$ as the empty set $\emptyset$.
\subsection{Application 1}
Using Theorem A, I can prove the following evaluations:
\begin{theorem}
Let $r\in\mathbb{Z}_{\ge0}$ and $s\in\mathbb{Z}_{\ge1}$. Then the sums
\begin{equation}
\begin{aligned}
\sum_{i=0}^{r}(-1)^{r-i}\left\{\binom{s-2+i}{i}+\binom{s-1+i}{i}\right\}
\zeta_{-}(\{1\}^{r-i},2s+i),
\end{aligned}
\end{equation}
\begin{equation}
\begin{aligned}
\sum_{i=0}^{r}(-1)^{r-i}2^{r-i+1}\left\{\binom{2s+i-1}{i}-\binom{2s+i-1}{i-1}\right\}
\zeta_{-}(\{1\}^{r-i},2s+i)
\end{aligned}
\end{equation}
can be expressed in $\mathbb{Q}$-polynomials of the Riemann zeta values 
$\zeta(k):=\sum_{m=1}^{\infty}m^{-k}$ $(k\in\mathbb{Z}_{\ge2})$. 
\end{theorem}
\begin{remark}
I first note the following two evaluations: 
\begin{equation}
\begin{aligned}
\zeta_{-}(\{1\}^{m},2)
=
&(-1)^{m}\zeta(m+2)+(-1)^{m}2\frac{(\log2)^{m+2}}{(m+2)!}\\
&+(-1)^{m+1}\sum_{k=0}^{m+2}\mathrm{Li}_k(1/2)\frac{(\log2)^{m+2-k}}{(m+2-k)!}
\end{aligned}
\end{equation}
($m\in\mathbb{Z}_{\ge0}$; Borwein, Bradley, and 
Broadhurst \cite[Identity (69), Section 6]{bbb}), where $\mathrm{Li}_{k}(z):=\sum_{m=1}^{\infty}z^{m}m^{-k}$ is the polylogarithm; and 
\begin{equation}
\begin{aligned}
\zeta_{-}(1,3)
=
&-2\mathrm{Li}_4(1/2)-\frac{1}{12}(\log2)^4+\frac{15}{8}\zeta(4)\\
&-\frac{7}{4}\zeta(3)\log2+\frac{1}{2}\zeta(2)(\log2)^2
\end{aligned}
\end{equation}
(Borwein, Borwein, and Girgensohn \cite[p.~291, line 10]{bbg}). 
I note that, while the evaluations (5) and (6) contain $\mathrm{Li}_{k}(1/2)$ ($k\in\mathbb{Z}_{\ge1}$; $\log2=\mathrm{Li}_{1}(1/2)$), 
the evaluations of the sums (3) and (4) in Theorem 1 do not contain 
$\mathrm{Li}_{k}(1/2)$. In fact, the sums (3) and (4) can be evaluated only 
by $\zeta(k)$ ($=\mathrm{Li}_{k}(1)$). From this fact, I think that the sums (3) and (4) are particularly interesting. Each of $\zeta_{-}(\{1\}^{m},n+1)$'s ($m,n\in\mathbb{Z}_{\ge0}$) can be evaluated 
by alternating Euler sums (see Borwein et al. \cite[Identity (28)]{bbb} and \cite[Theorem 9.3]{bbbl}). 
For MZVs, it is known that $\zeta(\{1\}^{m},n+2)$ ($m,n\in\mathbb{Z}_{\ge0}$) 
can be expressed in a $\mathbb{Q}$-polynomial of $\zeta(k)$ 
$(k\in\mathbb{Z}_{\ge2})$ (see \cite{ohza2001} and \cite{z}). 
\end{remark}
\par 
To prove Theorem 1, I need two lemmas. 
\begin{lemma}
The following two identities hold$:$
\par
$(i)$ 
\begin{equation}
\begin{aligned}
&\sum_{i=0}^{r}(-1)^{r-i}\left\{\binom{s-2+i}{i}+\binom{s-1+i}{i}\right\}
\zeta_{-}(\{1\}^{r-i},2s+i)\\
=
&\sum_{\begin{subarray}{c}
r_1+\cdots+r_s= r\\
       r_i\in\mathbb{Z}_{\ge0}
      \end{subarray}}
\zeta^{\star}(\{r_i+2\}_{i=1}^{s})
\end{aligned}
\end{equation}
for all $r\in\mathbb{Z}_{\ge0}$, $s\in\mathbb{Z}_{\ge1}.$
\par
$(ii)$ 
\begin{equation}
\begin{aligned}
&\sum_{i=0}^{r}(-1)^{r-i}2^{r-i+1}\left\{\binom{2s+i-1}{i}-\binom{2s+i-1}{i-1}\right\}\zeta_{-}(\{1\}^{r-i},2s+i)\\
=
&
\sum_{i=0}^{r}
\frac{(-1)^i}{i!}\frac{\mathrm{d}^{i}}{\mathrm{d}\alpha^{i}}
\left(\frac{\Gamma(\alpha)^2}{\Gamma(2\alpha-1)}\right){\Bigl|}_{\alpha=1}\\
&\times
\sum_{\begin{subarray}{c}
r_1+\cdots+r_s= r-i\\
       r_j\in\mathbb{Z}_{\ge0}
      \end{subarray}}
\left\{\prod_{j=1}^{s}(r_j+1)\right\}
\zeta^{\star}(\{r_j+2\}_{j=1}^{s})
\end{aligned}
\end{equation}
for all $r\in\mathbb{Z}_{\ge0}$, $s\in\mathbb{Z}_{\ge1}$, where $\Gamma(z)$ is 
the gamma function.
\end{lemma}
\begin{proof}
Taking $\alpha=\beta=1$ in \cite[Theorem 2.13 (iii)]{i4}, we have (7).
The proof of (8) is as follows: Taking $a=2\alpha$ and $b_i=c_i=\alpha$ ($i = 1, \ldots, s+1$) ($s\in\mathbb{Z}_{\ge1}$, $\alpha\in\mathbb{C}$ with 
$s+1/2>\mathrm{Re}(\alpha)>1/2$) in Theorem A, we have 
\begin{equation*}
\begin{aligned}
&2\sum_{m=0}^{\infty}\frac{(2\alpha-1)_{m+1}}{m!}\frac{(-1)^m}{(m+\alpha)^{2s+1}}\\
=
&
\frac{\Gamma(\alpha)^2}{\Gamma(2\alpha-1)}
\sum_{0{\le}m_1{\le}\cdots{\le}m_s<\infty}
\prod_{i=1}^{s}
\frac{1}{(m_i+\alpha)^2}.
\end{aligned}
\end{equation*}
Differentiating both sides of this identity $r$ times at $\alpha=1$ and using the identity
\begin{equation*} 
\frac{1}{r!}\frac{\mathrm{d}^r}{\mathrm{d}w^r}(w)_m
=(w)_m
\sum_{0{\le}m_1<\cdots<m_r<m}
\prod_{i=1}^{r}\frac{1}{m_i+w}
\end{equation*}
($r,m\in\mathbb{Z}_{\ge0}$), we obtain (8).
\end{proof}
\begin{lemma}
Let $k\in\mathbb{Z}_{\ge1}$, $q,r\in\mathbb{Z}_{\ge0}$, $s\in\mathbb{Z}_{\ge2}$, and let 
$f(r_1,\ldots,r_k)$ be a symmetric function with $k$ variables. 
Then the sum
\begin{equation}
\sum_{\begin{subarray}{c}
r_1+\cdots+r_k= r\\
       r_i\in\mathbb{Z}_{\ge0}
      \end{subarray}}
f(r_1,\ldots,r_k)
\zeta^{\star}(qr_1+s,\ldots,qr_k+s)
\end{equation}
can be expressed in a $\mathbb{Q}$-linear combination of 
$f(r_1,\ldots,r_k)\{\prod_{i=1}^{\ell}\zeta(m_i)\}$ $(m_i\in\mathbb{Z}_{\ge2})$.
\end{lemma}
\begin{proof}
Let $\mathfrak{S}_k$ be the symmetric group of degree $k$. Then we have 
\begin{equation}
\begin{aligned}
&\sum_{\sigma\in\mathfrak{S}_k}
\sum_{\begin{subarray}{c}
r_{\sigma(1)}+\cdots+r_{\sigma(k)}= r\\
       r_{\sigma(i)}\in\mathbb{Z}_{\ge0}
      \end{subarray}}
f(r_{\sigma(1)},\ldots,r_{\sigma(k)})
\zeta^{\star}(qr_{\sigma(1)}+s,\ldots,qr_{\sigma(k)}+s)\\
=
&\sum_{\begin{subarray}{c}
r_1+\cdots+r_k= r\\
       r_i\in\mathbb{Z}_{\ge0}
      \end{subarray}}
f(r_1,\ldots,r_k)
\sum_{\sigma\in\mathfrak{S}_k}
\zeta^{\star}(qr_{\sigma(1)}+s,\ldots,qr_{\sigma(k)}+s)
\end{aligned}
\end{equation}
for $k\in\mathbb{Z}_{\ge1}$, $q,r\in\mathbb{Z}_{\ge0}$, $s\in\mathbb{Z}_{\ge2}$. 
Taking $i_j=qr_j+s$ ($j=1,\cdots,k$) in \cite[Theorem 2.1]{h}, 
we see that the inner sum on the right-hand side of (10) can be 
written as a $\mathbb{Q}$-polynomial of $\zeta(m)$ ($m\in\mathbb{Z}_{\ge2}$). 
Thus the right-hand side becomes a $\mathbb{Q}$-linear combination of 
$f(r_1,\ldots,r_k)\{\prod_{i=1}^{\ell}\zeta(m_i)\}$. 
On the other hand, since $|\mathfrak{S}_k|=k!$, 
where $|\mathfrak{S}_k|$ denotes the number of elements of $\mathfrak{S}_k$, 
the left-hand side of (10) can be rewritten as
\begin{equation*}
k!
\sum_{\begin{subarray}{c}
r_1+\cdots+r_k= r\\
       r_i\in\mathbb{Z}_{\ge0}
      \end{subarray}}
f(r_1,\ldots,r_k)
\zeta^{\star}(qr_1+s,\ldots,qr_k+s),
\end{equation*}
which is (9). This completes the proof.
\end{proof}
\begin{remark}
Hoffman proved his evaluation \cite[Theorem 2.1]{h} explicitly; 
therefore the sum (9) can also be evaluated like that.
\end{remark}
\begin{proof}[Proof of Theorem $1$]
The case $q=1$, $s=2$, $f(r_1,\ldots,r_k)=1$ of Lemma 2 shows that 
the right-hand side of (7) 
can be expressed in a $\mathbb{Q}$-linear combination of $\prod_{i=1}^{\ell}\zeta(m_i)$, and this proves the assertion for (3). 
Similarly, the case $q=1$, $s=2$, $f(r_1,\ldots,r_k)=\prod_{i=1}^{k}(r_i+1)$ 
of Lemma 2 shows that the inner sum on the right-hand side of (8) 
can be expressed in a $\mathbb{Q}$-linear combination of $\{\prod_{i=1}^{k}(r_i+1)\}\{\prod_{i=1}^{\ell}\zeta(m_i)\}$. As regards the differential coefficient 
\begin{equation} 
\frac{1}{i!}\frac{\mathrm{d}^{i}}{\mathrm{d}\alpha^{i}}
\left(\frac{\Gamma(\alpha)^2}{\Gamma(2\alpha-1)}\right){\Bigl|}_{\alpha=1}
\quad\quad (i\in\mathbb{Z}_{\ge0}),
\end{equation}
it also has such an expression. 
Indeed, using the expansion of the gamma function
\begin{equation*}
\Gamma(\alpha)=\exp\left(-\gamma(\alpha-1)+\sum_{n=2}^{\infty}(-1)^n\frac{\zeta(n)}{n}(\alpha-1)^{n}\right)
\end{equation*}
for all $\alpha\in\mathbb{C}$ such that $|\alpha-1|<1$, where $\gamma$ is Euler's constant (see, e.g., \cite[p.~38]{ak}, \cite[Chapter XII]{whiwat}, and \cite{ohza2001}), we have 
\begin{equation*}
\frac{\Gamma(\alpha)^2}{\Gamma(2\alpha-1)}
=\exp\left(\sum_{n=2}^{\infty}(-1)^{n+1}\frac{\zeta(n)}{n}(2^n-2)(\alpha-1)^{n}\right)
\end{equation*}
for $|\alpha-1|<1/2$, and this gives an expression of (11) by 
a $\mathbb{Q}$-polynomial of $\zeta(k)$. 
Therefore, combining all the expressions above of the right-hand side of (8), 
we obtain a proof of the assertion for (4).
\end{proof}
Using Theorem A, I can prove also the following evaluation, which is similar to (7):
\begin{theorem}
The identity
\begin{equation}
\begin{aligned}
&\sum_{i=0}^{r}\left\{\binom{s-2+i}{i}+\binom{s-1+i}{i}\right\}
\zeta_{-}^{\star}(\{1\}^{r-i},2s+i)\\
=
&\sum_{\begin{subarray}{c}
r_1+\cdots+r_s= r\\
       r_i\in\mathbb{Z}_{\ge0}
      \end{subarray}}
(r_1+1)\zeta^{\star}(\{r_i+2\}_{i=1}^{s})
\end{aligned}
\end{equation}
holds for all $r\in\mathbb{Z}_{\ge0}$, $s\in\mathbb{Z}_{\ge1}$.
\end{theorem}
\begin{proof}
Taking $a=\alpha+1$, $b_1=c_1=1$, $b_i=\alpha$, $c_i=1$ ($i = 2, \ldots, s+1$) ($s\in\mathbb{Z}_{\ge1}$, $\alpha\in\mathbb{C}$ with $\mathrm{Re}(\alpha)>0$) 
in Theorem A, we have 
\begin{equation*}
\begin{aligned}
&\sum_{m=0}^{\infty}\frac{m!}{(\alpha)_{m+1}}\frac{2m+\alpha+1}{(m+\alpha)^{s}(m+1)^{s}}(-1)^m\\
=
&\sum_{0\le m_1\le \cdots{\le}m_s<\infty}
\frac{1}{(m_1+\alpha)^2}
\left\{\prod_{i=2}^{s}\frac{1}{(m_i+\alpha)(m_i+1)}\right\}.
\end{aligned}
\end{equation*}
Differentiating both sides of this identity $r$ times at $\alpha=1$ and using the identity
\begin{equation*} 
\frac{(-1)^r}{r!}\frac{\mathrm{d}^r}{\mathrm{d}w^r}\left(\frac{1}{(w)_{m+1}}\right)
=\frac{1}{(w)_{m+1}}
\sum_{0{\le}m_1{\le}\cdots{\le}m_r{\le}m}
\prod_{i=1}^{r}\frac{1}{m_i+w}
\end{equation*}
($r,m\in\mathbb{Z}_{\ge0}$), we obtain (12).
\end{proof}
\begin{remark}
(i) Taking $\alpha=\beta=\gamma=1$ in \cite[Theorem 2.13 (ii)]{i4}, 
I have another expression of the left-hand side of (12): 
\begin{equation*}
\textrm{The left-hand side of (12)}
=
\sum_{\begin{subarray}{c}
r_1+\cdots+r_s= r\\
       r_i\in\mathbb{Z}_{\ge0}
      \end{subarray}}
\zeta^{\star}(\{\{1\}^{r_i}, 2\}_{i=1}^{s})
\end{equation*}
($r\in\mathbb{Z}_{\ge0}$, $s\in\mathbb{Z}_{\ge1}$), where 
$\{\{1\}^{r_i}, 2\}_{i=1}^{s}:=\{1\}^{r_1}, 2,\ldots,\{1\}^{r_i}, 2,\ldots,\{1\}^{r_s}, 2$. 
\par 
(ii) For $s=1$, the identities (7) and (12) become 
\begin{equation}
\sum_{i=0}^{r}(-1)^{r-i}
(\delta_{i0}+1)\zeta_{-}(\{1\}^{r-i},2+i)
=
\zeta(r+2),
\end{equation}
\begin{equation}
\sum_{i=0}^{r}(\delta_{i0}+1)\zeta_{-}^{\star}(\{1\}^{r-i},2+i)
=
(r+1)\zeta(r+2)
\end{equation}
($r\in\mathbb{Z}_{\ge0}$), respectively, where $\delta_{ij}$ denotes 
Kronecker's delta, i.e.,  $\delta_{ij}=1$ if $i=j$ and $\delta_{ij}=0$ if $i{\neq}j$. 
I note that there is a similarity between (13), (14) and the identity for MZSVs of Aoki and Ohno \cite[Corollary 1]{ao}.
\end{remark}
\subsection{Application 2}
As shown in \cite{i4}, the following identity can be derived from Theorem A: 
\begin{equation}
\begin{aligned}
&\sum_{i=0}^{k}2^{k-i}
\sum_{\begin{subarray}{c}
k_1+\cdots+k_s=k-i\\
k_j\in\mathbb{Z}_{\ge0}
\end{subarray}}
\zeta^{\star}(i+k_1+2,\{k_j+2\}_{j=2}^{s})\\
=
&2^{1+k}\binom{k+s-1}{k}(1-2^{1-k-2s})\zeta(k+2s)
\end{aligned}
\end{equation}
for all $k\in\mathbb{Z}_{\ge0}$, $s\in\mathbb{Z}_{\ge1}$. (For details, 
see \cite[Theorem 2.16 and its proof, Remark 8]{i4}.) 
Using this identity, I can prove the following evaluation of MZSVs:
\begin{theorem}
The identity
\begin{equation}
\begin{aligned}
&\zeta^{\star}(4,\{2\}^{s-1})\\
=
&c(s)\zeta^{\star}(\{2\}^{s+1})
-\frac{2}{3}\sum_{\begin{subarray}{c}
k_1+k_2+k_3=s-2\\
k_j\in\mathbb{Z}_{\ge0}
\end{subarray}}
(2+\delta_{0k_1})\zeta^{\star}(\{2\}^{k_1},3,\{2\}^{k_2},3,\{2\}^{k_3})
\end{aligned}
\end{equation}
holds for all $s\in\mathbb{Z}_{\ge1}$, where $\delta_{ij}$ denotes Kronecker's delta defined under 
$(14)$ and 
\begin{equation}
\begin{aligned}
c(s)
:=
\frac{2s(s+1)}{3}+\frac{2}{3}\sum_{i=2}^{s+1}\binom{2(s+1)}{2i}
\frac{B_{2(s+1-i)}B_{2i}}{B_{2(s+1)}}
\frac{(1-2^{1-2i})}{(1-2^{1-2(s+1)})}
\end{aligned}
\end{equation}
$(s\in\mathbb{Z}_{\ge1})$. Here $B_n$ is the $n$-th Bernoulli number. 
\end{theorem}
\begin{remark}
The identity (16) is an expression of $\zeta^{\star}(4,\{2\}^{s-1})$ in terms of 
the $\{2,3\}$-basis of the $\mathbb{Q}$-vector space of MZSVs. For the $\{2,3\}$-basis, see \cite[Section 3]{ikoo}. 
\end{remark}
\begin{proof}[Proof of Theorem $3$]
To prove (16), we use the case $k=2$ of (15): 
\begin{equation}
\begin{aligned}
&3\zeta^{\star}(4,\{2\}^{s-1})+2^2\sum_{i=1}^{s}\zeta^{\star}(\{2\}^{i-1},4,\{2\}^{s-i})\\
&+2\sum_{\begin{subarray}{c}
k_1+k_2+k_3=s-2\\
k_j\in\mathbb{Z}_{\ge0}
\end{subarray}}
(2+\delta_{0k_1})\zeta^{\star}(\{2\}^{k_1},3,\{2\}^{k_2},3,\{2\}^{k_3})\\
=
&{2^2}s(s+1)(1-2^{-1-2s})\zeta(2+2s)
\end{aligned}
\end{equation}
for $s\in\mathbb{Z}_{\ge1}$. By using the identity
\begin{equation}
2(1-2^{1-2s})\zeta(2s)=\zeta^{\star}(\{2\}^s)
\end{equation}
($s\in\mathbb{Z}_{\ge1}$; see \cite{zl2}), the right-hand side of (18) can be rewritten as
\begin{equation}
\begin{aligned}
{2^2}s(s+1)(1-2^{-1-2s})\zeta(2+2s)=2s(s+1)\zeta^{\star}(\{2\}^{s+1})
\end{aligned}
\end{equation}
for $s\in\mathbb{Z}_{\ge1}$. 
As regards the first sum on the left-hand side of (18), by using (19) and Euler's formula 
\begin{equation*}
\zeta(2s)=(-1)^{s-1}\frac{B_{2s}}{(2s)!}\frac{(2\pi)^{2s}}{2} 
\quad\quad (s\in\mathbb{Z}_{\ge1}), 
\end{equation*}
it can be rewritten as
\begin{equation}
\begin{aligned}
&\sum_{i=1}^{s}\zeta^{\star}(\{2\}^{i-1},4,\{2\}^{s-i})\\
=
&\sum_{i=0}^{s-1}\zeta(2(i+2))\zeta^{\star}(\{2\}^{s-1-i}) \quad\quad (\text{see the proof below})\\
=
&2\sum_{i=0}^{s-1}(1-2^{1-2(s-1-i)})\zeta(2(i+2))\zeta(2(s-1-i))\\
=
&\left(\sum_{i=0}^{s-1}
(1-2^{1-2(s-1-i)})
\frac{B_{2(i+2)}B_{2(s-1-i)}}{(2(i+2))!(2(s-1-i))!}\right)
(-1)^{s-1}\frac{(2\pi)^{2(s+1)}}{2}\\
=
&-\left(\sum_{i=0}^{s-1}
(1-2^{1-2(s-1-i)})
\frac{B_{2(i+2)}B_{2(s-1-i)}}{(2(i+2))!(2(s-1-i))!}\right)
\frac{(2(s+1))!}{B_{2(s+1)}}\zeta(2(s+1))\\
=
&-\frac{1}{2}
\left(\sum_{i=0}^{s-1}\binom{2(s+1)}{2(i+2)}\frac{B_{2(i+2)}B_{2(s-1-i)}}{B_{2(s+1)}}
\frac{(1-2^{1-2(s-1-i)})}{(1-2^{1-2(s+1)})}\right)
\zeta^{\star}(\{2\}^{s+1})
\end{aligned}
\end{equation}
for $s\in\mathbb{Z}_{\ge1}$, where $\zeta(0)=-1/2$ and $\zeta^{\star}(\emptyset)=1$. 
Therefore, substituting (20) and (21) into (18), we obtain (16). The first identity of (21) can be proved in a way similar to that used for MZVs in 
\cite[pp.~8 and 89]{ak}. Indeed, using the harmonic product of MZSVs 
(see \cite{ikoo} and \cite{mu}), we have 
\begin{equation*}
\begin{aligned}
\zeta(t)\zeta^{\star}(\{2\}^{s-1})
&=
\sum_{i=1}^{s}\zeta^{\star}(\{2\}^{i-1},t,\{2\}^{s-i})
-\sum_{i=1}^{s-1}\zeta^{\star}(\{2\}^{i-1},t+2,\{2\}^{s-1-i}),\\
\zeta(t+2)\zeta^{\star}(\{2\}^{s-2})
&=
\sum_{i=1}^{s-1}\zeta^{\star}(\{2\}^{i-1},t+2,\{2\}^{s-1-i})
-\sum_{i=1}^{s-2}\zeta^{\star}(\{2\}^{i-1},t+4,\{2\}^{s-2-i}),\\
&\hspace{3cm} \vdots\\
\zeta(t+2s-4)\zeta^{\star}(2)
&=
\zeta^{\star}(t+2s-4,2)+\zeta^{\star}(2,t+2s-4)-\zeta(t+2s-2)
\end{aligned}
\end{equation*}
for $s,t\in\mathbb{Z}_{\ge2}$. 
Further, adding up each side of all these identities, we obtain 
\begin{equation*}
\sum_{i=0}^{s-1}\zeta(2i+t)\zeta^{\star}(\{2\}^{s-1-i})
=
\sum_{i=1}^{s}\zeta^{\star}(\{2\}^{i-1},t,\{2\}^{s-i})
\end{equation*}
for $s\in\mathbb{Z}_{\ge1}$, $t\in\mathbb{Z}_{\ge2}$. 
The first identity of (21) is the case $t=4$ of this identity. This completes the proof of Theorem 3.
\end{proof}
\begin{remark}
(i) The identity (15) contains both (19) and the identity for $\zeta(2s+1)$ of 
\cite[Theorem 2]{ikoo} as special cases: $k=0$ and $k=1$, respectively. 
Compare this with \cite[Proof of Theorem 2]{ikoo}. 
These two cases give a $\{2,3\}$-basis expression of $\zeta(k)$ for any $k\in\mathbb{Z}_{\ge2}$ (see \cite[Theorem 2]{ikoo}). 
\par 
(ii) Taking $\alpha=1$ in \cite[Theorem 2.16]{i4}, I have the following general form of (15), which was proved in 2014: 
\begin{equation*}
\begin{aligned}
&\sum_{i=0}^{k}2^{k-i}\binom{i+r}{i}
\sum_{\begin{subarray}{c}
k_1+\cdots+k_{s}=k-i\\
k_j\in\mathbb{Z}_{\ge0}
\end{subarray}}
\zeta^{\star}(i+k_1+r+2, \{k_j+2\}_{j=2}^{s})\\
=
&\sum_{i=0}^{r}
\sum_{\begin{subarray}{c}
k_1+\cdots+k_{i+2}=k\\
r_1+\cdots+r_{i+1}=r+1\\
k_j\in\mathbb{Z}_{\ge0}, r_j\in\mathbb{Z}_{\ge1}
\end{subarray}}
2^{i+1+k_{i+2}}
\left\{\prod_{j=1}^{i}\binom{k_j+r_j-1}{k_j}\right\}
\binom{k_{i+1}+r_{i+1}-2}{k_{i+1}}
\binom{k_{i+2}+s-1}{k_{i+2}}\\
&
\times
\zeta_{-}(\{k_j+r_j\}_{j=1}^{i},k_{i+1}+k_{i+2}+r_{i+1}+2s-1)
\end{aligned}
\end{equation*}
for all $k,r\in\mathbb{Z}_{\ge0}$, $s\in\mathbb{Z}_{\ge1}$. 
It is interesting to find applications of this general form, and also to obtain generalizations of (16) concerning 
the $\{2,3\}$-basis expression of MZSVs. 
\end{remark}
\begin{remark}
This was obtained in 2018. The coefficient (17) has the following closed form: 
\begin{equation}
\begin{aligned}
c(s)
=&
\frac{2}{3}\{s(s-1)-1+(1-2^{1-2(s+1)})^{-1}\}\\
&-\frac{1}{18}\frac{(s+1)(2s+1)}{(1-2^{1-2(s+1)})}\frac{B_{2s}}{B_{2(s+1)}}
\end{aligned}
\end{equation}
for all $s\in\mathbb{Z}_{\ge1}$. Indeed, taking $n=s+1$ in the identity on 
page 154, line 11 of \cite{n} and using $B_0=1$, $B_2=1/6$, we have the following identity for the Bernoulli numbers: 
\begin{equation*}
\begin{aligned}
&\sum_{i=2}^{s+1}\binom{2(s+1)}{2i}B_{2(s+1-i)}B_{2i}(1-2^{1-2i})\\
=
&\{1-(2s+1)(1-2^{1-2(s+1)})\}B_{2(s+1)}-\frac{(s+1)(2s+1)}{12}B_{2s}
\end{aligned}
\end{equation*}
for $s\in\mathbb{Z}_{\ge1}$. Dividing both sides of this identity by 
$B_{2(s+1)}(1-2^{1-2(s+1)})$ and substituting the result into (17), we obtain (22). 
\end{remark}
\begin{note}
I take this opportunity to give some notes about my papers \cite{i1} and \cite{i4} 
and some of my related research. 
\par 
(i) I note that a revised manuscript of \cite{i1} which contains Remark 2.7 was first submitted to the journal on August 4, 2009. 
This is an additional explanation for the addition of Remark 2.7 to \cite{i1}. 
See also an explanation at the beginning of \cite[Remark 2.7]{i1} 
and arXiv:0908.2536v1. 
\par 
(ii) I note that the identities (27), (28), (57), (58) of \cite{i4} were proved in 2014 
(before October 15, 2014). (These new identities were discovered by using \cite[Proposition 1 (i)]{kr}.) I had already written them in 
unpublished manuscripts in 2014 (before October 15, 2014). 
For instance, the identity (28) is written in the expanded version of \cite{i2011} dated September 17, 2014, though it was already written in 
another unpublished manuscript of mine before September 17, 2014. 
It is interesting to generalize my identities above and my previous results \cite[Addendum]{i2011}, \cite[Identities (29) and (30)]{i4}. 
For non-trivial two-parameter generalizations of them, 
see \cite[Corollaries 2.10, 2.12, and Theorem 2.14]{i4}. These generalizations of 
mine and my result \cite[Theorem 2.15]{i4} were already 
written in my manuscripts submitted to journals on March 9, May 1 and 16, 2015. I note that I had already obtained many of the results of \cite{i4} 
before October 15, 2014, and had already written them in unpublished manuscripts before October 15, 2014. For instance, I obtained the identity (63) of \cite{i4} in February 2013, and wrote it in an unpublished manuscript in February 2013. 
Moreover, I wrote about this identity of mine and an outline of its proof 
(\cite[Proof of Theorem 2.13 (iv)]{i4}) in a research plan document in February 2013, and sent it to my then principal academic advisor by e-mail in February 2013. I note that the identity (63) of \cite{i4} is a non-trivial one-parameter generalization of the identity for multiple zeta-star values of Aoki and Ohno \cite[Theorem 1]{ao} (see \cite[Remark 6 (ii)]{i4}). 
In addition, my proof of the identity (63) gives a new hypergeometric proof of \cite[Theorem 1]{ao}, 
which is based on Andrews' hypergeometric identity \cite[Theorem 4]{an}. For another instance, I obtained the identities (8) and (26) of \cite{i4} in April 2014, and wrote them in an unpublished manuscript in April 2014. 
I note that my paper \cite{i4} is a revised version of my manuscript submitted to the Nagoya Mathematical Journal on March 9, 2015. (This manuscript was rejected on April 19, 2015.) 
A preprint of the manuscript submitted on March 9, 2015 was distributed on February 12, 2015. 
I had already written most of the contents of \cite{i4} in the manuscript and the preprint; hence it is already shown there that the one-, two-, three- and four-parameter multiple series of \cite{i4} satisfy various non-trivial and interesting relations. This shows the potentiality of these multiple series. 
I began the study of the three- and four-parameter multiple series in 2014. 
The study is based on my prior works on one- and two-parameter multiple series. In 2013, I proved a non-trivial one- and two-parameter generalization of my identity (29) of \cite{i4}. The former was written in my preprint distributed on October 28, 2013 and the latter in 
my preprint distributed on December 16, 2013. (These preprints are earlier versions of \cite{i4}.) 
The identity (8) of \cite{i4} is an extension of them to a three-parameter multiple series. Indeed, the case $\alpha=\beta=\gamma$ and the case $\beta=\gamma$ of (8) are the above one- and two-parameter generalization, respectively. 
I note that the problem posed at the end of $\textbf{(R1)}$ of \cite[Section 3]{i4} was already written in the above preprint distributed on December 16, 2013. 
Besides this, I wrote also a non-trivial one- (resp. two-) parameter generalization of my identity (30) of \cite{i4} in the above preprint distributed on October 28, 2013 (resp. December 16, 2013), though they were not written in \cite{i4}. 
In the manuscript submitted to Nagoya Math. J. on March 9, 2015 and its preprint distributed on February 12, 2015, I wrote the following identity, which is not written in \cite{i4}: 
\begin{equation*}
\begin{aligned}
&
\frac{\Gamma(\alpha)\Gamma(\beta)}{\Gamma(\alpha+\beta)}
\sum_{0\le m_1\le \cdots{\le}m_s<\infty}
\frac{(\alpha)_{m_1}}{{m_1}!}
\frac{(\beta)_{m_1}}{(\alpha+\beta)_{m_1}}
\frac{1}{(m_1+\gamma)^{r+1}}
\left\{\prod_{i=2}^{s}\frac{1}{(m_i+\alpha)(m_i+\beta)}\right\}\\
=
&\sum_{i=0}^{r}
\sum_{\begin{subarray}{c}
r_1+\cdots+r_{i+1}=r+1\\
1{\le}r_j\in\mathbb{Z}
\end{subarray}}
\sum_{0{\le}m_1<\cdots<m_{i+1}<\infty}
(-1)^{m_{i+1}}\frac{(\alpha+\beta-\gamma)_{m_{i+1}}}{(\gamma)_{m_{i+1}}}\\
&\times\left\{\prod_{j=1}^{i}
\frac{2m_j+\alpha+\beta}{(m_j+\gamma)^{r_j}(m_j+\alpha+\beta-\gamma)}
\right\}\frac{2m_{i+1}+\alpha+\beta}{(m_{i+1}+\gamma)^{r_{i+1}}(m_{i+1}+\alpha)^s(m_{i+1}+\beta)^s}
\end{aligned}
\end{equation*}
for all $0{\le}r\in\mathbb{Z}$, $1{\le}s\in\mathbb{Z}$, and $\alpha,\beta,\gamma\in\mathbb{C}$ such that $\mathrm{Re}\,\alpha>0$, $\mathrm{Re}\,\beta>0$, $\mathrm{Re}\,(\alpha+\beta)>\mathrm{Re}\,\gamma>0$, 
$2s-1>\mathrm{Re}\,(\alpha+\beta-2\gamma)$, where $(a)_m$ is the Pochhammer symbol; $\Gamma(z)$ is the gamma function. (This identity was also derived from \cite[Proposition 1 (i)]{kr}.) The above one- (resp. two-) parameter generalization of my identity (30) of \cite{i4} is the case $\alpha=\beta=\gamma$ 
(resp. the case $\beta=\gamma$) of this identity. I note that I proved 
the case $\alpha=\beta=\gamma$ in August 2012, and wrote it in an unpublished manuscript in August 2012. It was written also in the expanded version of \cite{i2011} submitted to a journal on March 5, 2013, which was rejected on May 9, 2013. For related results, see also \cite[(\textbf{A1}) and (\textbf{A2})]{i2011}. The identities (26), (45), (55) of \cite{i4} are variations of the above generalizations of my identities (29) and (30) of \cite{i4}. 
(The case $\beta=1$ of (55) was written in an unpublished manuscript of mine in April 2014.) I note that the identities (2), (10), (24), (43), (48), (66), (67), (70), (71) of \cite{i4} and the identity on page 731 of \cite{i4} were already written in my manuscript submitted to Nagoya Math. J. on March 9, 2015 and in its preprint distributed on February 12, 2015. 
Besides these identities, the following contents of \cite{i4} were also written in 
the manuscript and the preprint: the identities (14), (60), (62), (63), (77), 
(78), (80), (85), the case $k=0$ of (15), the case $\beta=\gamma$ of (61), 
the remarks $\textbf{(R2)}$, $\textbf{(R3)}$, and the problem posed at the end 
of $\textbf{(R1)}$ of \cite{i4}. 
The identity (44) of \cite{i4} was written in the manuscript of \cite{i4} submitted to a journal on May 1, 2015, which was rejected on May 11, 2015. 
The identities (33) and (34) were written in the manuscript of \cite{i4} submitted to a journal on May 16, 2015, which was rejected on January 23, 2016. 
\par 
(iii) In the manuscript submitted to Nagoya Math. J. on March 9, 2015 and its preprint distributed on February 12, 2015, I proved the following duality formula for the standard multiple Hurwitz zeta values 
by using an iterated integral representation of them: 
\begin{equation}
\sum_{0\le m_1< \cdots<m_{p}<\infty}
\prod_{i=1}^{p}\frac{1}{(m_i+\alpha)^{k_i}}
=
\sum_{0\le m_1< \cdots<m_{q}<\infty}
\frac{(m_{q}+1)!}{(\alpha)_{m_{q}+1}}
\left\{\prod_{i=1}^{q}\frac{1}{(m_i+1)^{k^{'}_i}}\right\}
\end{equation}
for all $1{\le}p,k_i\in\mathbb{Z}$ ($i=1,\ldots,p-1$), 
$2{\le}k_p\in\mathbb{Z}$, and $\alpha\in\mathbb{C}$ such that $\mathrm{Re}(\alpha)>0$, where $(a)_m$ is the Pochhammer symbol; $(\{k^{'}_i\}_{i=1}^{q})$ is the dual index of $(\{k_i\}_{i=1}^{p})$. See \cite{ig2007} and also 
\cite[Lemmas 2.2 and 2.3]{i2009}. For a definition of the dual index, see, e.g., \cite{i2009}. I stated the duality formula (23) in my talk at a seminar on February 13, 2008 (see \cite[Acknowledgments on p.~578]{i2009}; I remember the audience of my talk very well). See also the note on (23) below. Moreover, in the manuscript and the preprint, I showed that a large class of identities for the standard multiple Hurwitz zeta values can be derived from 
(23) by differentiating both sides of (23) and using the identities 
\begin{equation*}
\begin{aligned}
&\frac{(-1)^r}{r!}\frac{\mathrm{d}^r}{\mathrm{d}\alpha^r}
\left(\frac{1}{(\alpha)_{m_q+1}}\right)\\
=&\frac{(-1)^r}{r!}\frac{\mathrm{d}^r}{\mathrm{d}\alpha^r}
\left(\frac{1}{(\alpha)_{m_1+1}}\prod_{i=2}^{q}\frac{(\alpha)_{m_{i-1}+1}}{(\alpha)_{m_i+1}}\right)\\
=&
\frac{1}{(\alpha)_{m_q+1}}
\sum_{\begin{subarray}{c}
r_1+\cdots+r_q= r\\
       r_i\in\mathbb{Z}_{\ge0}
      \end{subarray}}
\sum_{\begin{subarray}{c}
0{\le}m_{11}\le\cdots{\le}m_{1 r_1}{\le}m_1\\
m_1<m_{21}\le\cdots{\le}m_{2 r_2}{\le}m_2\\
\vdots\\
m_{q-1}<m_{q 1}\le\cdots{\le}m_{q r_{q}}{\le}m_q
\end{subarray}}
\prod_{i=1}^{q}\prod_{j=1}^{r_i}\frac{1}{m_{ij}+\alpha}
\end{aligned}
\end{equation*}
for all $1{\le}q\in\mathbb{Z}$, $0{\le}r\in\mathbb{Z}$, and $m_1,\ldots,m_q\in\mathbb{Z}$ such that $0\le{m_1}<\cdots<m_q$. 
See also \cite[Section 2]{i2009}. I note that, in the expanded version of \cite{i2011} submitted to a journal on March 5, 2013, I had already used identities for the Pochhammer symbol $(a)_m$ of this kind for the purpose of proving relations 
for multiple series. My results on the standard multiple Hurwitz zeta values stated above 
were written also in the manuscript of \cite{i4} submitted to a journal on May 16, 2015, which was rejected on January 23, 2016. My results are not written in the published version \cite{i4}, because, in 2015, I extended them to the multiple series (24) and (25) below, and wrote the extended results in other manuscripts, which were distributed as preprints in October 2015; 
\begin{equation}
\sum_{0\le m_1{<_1}\cdots{<_{p-1}}m_{p}<\infty}
\prod_{i=1}^{p}\frac{1}{(m_i+\alpha)^{k_i}},
\end{equation}
\begin{equation}
\begin{aligned}
\sum_{\begin{subarray}{c}0{\le}m_1<_{1}\cdots<_{p-1}m_p<\infty\end{subarray}}
\frac{{m_p}!}{(\alpha)_{m_p}}
\left\{\prod_{i=1}^{p}\frac{1}{(m_i+\alpha)^{a_i}(m_i+1)^{b_i}}\right\},
\end{aligned}
\end{equation}
where $1{\le}p\in\mathbb{Z}$; $a_i, b_i,k_i\in\mathbb{Z}$ such that $a_i+b_i,k_i\ge1$ ($i=1,\ldots,p-1$), $a_p+b_p,k_p\ge2$; 
$\alpha\in\mathbb{C}$ with $\mathrm{Re}(\alpha)>0$; the symbols $<_i$ 
($i=1,\ldots,p-1$) denote $<$ or $\le$; $<_0=\le$. 
See \cite{i2022} and also \cite[Remark 10]{i2020}. 
Revised versions of the above manuscripts distributed in October 2015 
were submitted to many journals in 2016--2021. 
For instance, one was submitted to a journal on October 30, 2017, which was 
rejected on December 24, 2018. Another one was submitted to a journal 
on March 9, 2019, which was rejected on August 13, 2020. 
I posted one of the revised versions on arXiv.org in 2022; see \cite{i2022}. 
The multiple series (25) is an extension of the one-parameter multiple series studied in Coppo \cite{c}, Coppo and Candelpergher \cite{cc}, \'{E}mery \cite{e}, Hasse \cite{h}, \cite[Examples]{ig2007}, \cite{i4}. 
For the duality formula (23), I note also the following: In my talk at a seminar on February 13, 2008 mentioned above, I pointed out a similarity between the multiple series on the right-hand side of (23) and the Newton series studied by Kawashima in \cite{kaw2009}. Indeed, as explained in my talk, the Newton series 
has the factor $(\alpha)_{m_q}/{m_q}!$ in its summand and the multiple series on 
the right-hand side has the inverse ${m_q}!/(\alpha)_{m_q}$. 
I think that it is interesting to study this similarity further. 
In 2015, I extended also the results of \cite{i2009} to my multiple series (26) below, and wrote also the extended results in the above manuscripts distributed in October 2015; 
\begin{equation}
\begin{aligned}
\sum_{\begin{subarray}{c}0{\le}m_1<_{1}\cdots<_{p-1}m_p<\infty\end{subarray}}
\frac{(\alpha)_{m_1}}{{m_1}!}\frac{{m_p}!}{(\alpha)_{m_p}}
\left\{\prod_{i=1}^{p}\frac{1}{(m_i+\alpha)^{a_i}(m_i+\beta)^{b_i}}\right\},
\end{aligned}
\end{equation}
where $p$, $a_i$, $b_i$, $\alpha$, $<_i$ are the same as those in (25); 
$\beta\in\mathbb{C}\setminus\{0,-1,-2,\ldots\}$. 
The study of this kind of two-parameter multiple series was originated by me in 
\cite{ig2007}. The multiple series (24)--(26) generalize the extended multiple zeta value studied by Fischler and Rivoal \cite{fr}, Kawashima \cite{kaw}, and Ulanskii \cite{ul}. The extended multiple zeta value and (24)--(26) allow taking both $<$ and $\le$ in their summations. I note that multiple series of this kind naturally appear as derivatives of hypergeometric series (see \cite{i4}). 
I studied more general multiple series than (24)--(26) in the manuscripts of 
\cite{i4} submitted on March 9 and May 16, 2015 and in the above preprint distributed on February 12, 2015. For details of the study, see \cite{i4}. 
\par 
(iv) I note also the following research of mine related to (26): I proved the cyclic sum formulas for my multiple series (27) and (28) below in October 2012 and April 2013, 
respectively, and wrote them in unpublished manuscripts in October 2012 and April 2013; 
\begin{equation}
\sum_{0\le m_1\le \cdots {\le}m_p<\infty}\frac{(\alpha)_{m_1}}{{m_1}!}
\frac{{m_p}!}{(\alpha)_{m_p}}
\frac{1}{(m_1+\beta)^{k_1}}
\left\{ \prod_{i=2}^{p}\frac{1}{(m_i+\alpha)(m_i+\beta)^{k_i-1}} \right\},
\end{equation}
\begin{equation}
\sum_{0\le m_1< \cdots<m_p<\infty}
\frac{(\alpha)_{m_1}}{{m_1}!}
\frac{{m_p}!}{(\alpha)_{m_p}}
\frac{1}{(m_1+\beta)^{k_1}}\left\{\prod_{i=2}^{p}\frac{1}{(m_i+\alpha)(m_i+\beta)^{k_i-1}}\right\},
\end{equation}
where $1{\le}p,k_i\in\mathbb{Z}$ ($i=1,\ldots,p-1$), $2{\le}k_p\in\mathbb{Z}$, 
$\alpha\in\mathbb{C}$ such that $\mathrm{Re}(\alpha)>0$, 
$\beta\in\mathbb{C}\setminus\{0,-1,-2,\ldots\}$; $(a)_m$ is the Pochhammer symbol. I note that one side of my cyclic sum formula for (28) consists of the multiple 
series 
\begin{equation*}
\sum_{0{\le}m_1< \cdots<m_p<\infty}
\frac{(\alpha)_{m_1}}{{m_1}!}
\frac{{m_p}!}{(\alpha)_{m_p}}
\left\{\prod_{i=1}^{p-1}\frac{1}{(m_i+\alpha)(m_i+\beta)^{k_i-1}}\right\}
\frac{1}{(m_p+\alpha)^2(m_p+\beta)^{k_p-2}}.
\end{equation*}
The above three multiple series are special cases of my multiple series (26). 
My cyclic sum formula for (27) gives a sum formula for (27). 
I wrote also this sum formula in an unpublished manuscript in October 2012. 
My cyclic sum and my sum formula for (27) were written in my manuscripts submitted to a journal on March 15, 2013 and February 17, 2016. 
The manuscript submitted on March 15, 2013 was distributed as a preprint on 
December 5, 2013. My cyclic sum formula for (28) was written in an unpublished manuscript of mine on April 9, 2013. I wrote about my cyclic sum formula 
for (27) and for (28) in my postdoctoral research plan documents submitted to 
a previous affiliation of mine (Graduate School of Mathematics, Nagoya University, Japan) in Februaries 2013 and 2014, respectively. 
Both my cyclic sum formulas and my sum formula were written in a preprint 
of mine distributed on February 12, 2015 and in my manuscript submitted to a journal on February 6, 2018, which was rejected on June 11, 2018. 
I submitted revised versions of this rejected manuscript to many journals in 2018--2021. In the above manuscript submitted on February 17, 2016 
(rejected on March 10, 2016), I posed the problem 
whether my multiple series (29) below satisfies the cyclic sum formula; 
\begin{equation}
\begin{aligned}
&\sum_{0{\le}m_1{\le}\cdots{\le}m_p<\infty}
\frac{(\alpha)_{m_1}}{{m_1}!}\frac{(\beta)_{m_1}}{(\gamma)_{m_1}}
\frac{{m_p}!}{(\alpha)_{m_p}}\frac{(\gamma)_{m_p}}{(\beta)_{m_p}}\\
&\times
\left\{\prod_{i=1}^{p}\frac{1}{(m_i+\alpha)^{a_i}(m_i+\beta)^{b_i}(m_i+\gamma)^{c_i}}\right\},
\end{aligned}
\end{equation}
where $1{\le}p\in\mathbb{Z}$; $a_i,b_i,c_i\in\mathbb{Z}$ ($i=1,\ldots,p$) such that $a_i+b_i+c_i\ge1$ ($i=1,\ldots,p-1$), 
$a_p+b_p+c_p\ge2$; $\alpha,\beta,\gamma\in\mathbb{C}\setminus\{0,-1,-2,\ldots\}$ such that $\mathrm{Re}(\alpha+\beta-\gamma)>0$; $(a)_m$ is the Pochhammer symbol. To be more precise, I posed the problem whether the following special case of my multiple series (29) satisfies a relation similar to the cyclic sum formula for multiple zeta-star values: 
\begin{equation*}
\begin{aligned}
&\sum_{0\le m_1{\le} \cdots {\le}m_p<\infty}\frac{(\alpha)_{m_1}}{{m_1}!}\frac{(\beta)_{m_1}}{(\gamma)_{m_1}}
\frac{{m_p}!}{(\alpha)_{m_p}}\frac{(\gamma)_{m_p}}{(\beta)_{m_p}}\\
&\times\frac{1}{(m_1+\gamma)^{k_1}}
\left\{\prod_{i=2}^{p}\frac{1}{(m_i+\alpha)(m_i+\beta)(m_i+\gamma)^{k_i-2}}\right\},
\end{aligned}
\end{equation*}
where $1{\le}k_i\in\mathbb{Z}$ ($i=1,\ldots,p-1$), $2{\le}k_p\in\mathbb{Z}$. (The indexes of this multiple series are the same as those of my multiple series (27) and (28).) I think that it is interesting to study this problem of mine. 
This problem was written also in my postdoctoral research plan 
document submitted to a previous affiliation of mine (Graduate School of Mathematics, Nagoya University, Japan) in February 2016. 
The multiple series (29) was discovered by me in 2014 by studying the hypergeometric identities of Andrews \cite[Theorem 4]{an}, Krattenthaler and Rivoal \cite[Proposition 1]{kr}. 
My previous remark \cite[Remark 2.7]{i1} played an essential role for discovering (29) and its identities. 
I wrote (29) and its non-trivial identities in an unpublished manuscript in April 2014. In May 2014, I extended (29) to a four-parameter multiple series, which is studied in \cite{i4} (see (1) of \cite{i4}) and in the expanded version of \cite{i2011} dated September 17, 2014. I had already written the four-parameter multiple series in an unpublished manuscript in May 2014. 
\par 
(v) In Notes 1 and 2 (ii)--(iv) of the present paper, I referred to research documents of mine 
written in August 2012--February 2018: unpublished manuscripts, submissions 
to journals, preprints, and (postdoctoral) research plan documents. I note that all of these research documents are electronic files made in a computer system of 
a previous affiliation of mine (Graduate School of Mathematics, Nagoya University, Japan) in August 2012--February 2018. 
(This can be confirmed by asking people who read the contents of 
my electronic files in 2012--2018.) In fact, in 2012--March 2018, 
I used the computer system whenever I wrote research documents. 
For instance, I used it when I wrote the unpublished manuscripts on 
my identities (27), (28), (57), (58) of \cite{i4} in 2014 and the unpublished manuscripts on my cyclic sum formula for (27) and for (28) in October 2012 
and April 2013. I note that multiple Hurwitz zeta values and one- and two-parameter 
multiple series have been my main research objects since 2006 
(see, e.g., \cite{ig2007}). My results on these research objects form 
the basis of my research on three- and four-parameter multiple series. 
The study of three- and four-parameter multiple series was originated by me in 2014. 
\par 
(vi) In \cite[Remark 7 (i)]{i4}, I gave a new proof of Hoffman's identity 
$\zeta(\{1\}^l, k+2)=\zeta(\{1\}^k, l+2)$ ($k,l\in\mathbb{Z}_{\ge0}$; 
\cite[Theorem 4.4]{h}), which is the duality formula for $\zeta(\{1\}^l,k+2)$. 
My proof is based on the hypergeometric identities \cite[Theorem 4]{an} and \cite[Proposition 1 (i)]{kr}; 
therefore it can be regarded as a hypergeometric proof of the duality formula. 
It is interesting to generalize my proof to a proof of the duality formula 
for all MZVs. 
\end{note}
\begin{corrections}
(i) Page 713, line 9 from the bottom: ``and a revised'' should be ``and is a revised''. 
(ii) Page 726, line 10 from the bottom: ``I remark that'' should be ``I note that''. 
(iii) Page 726, line 8 from the bottom: ``and my observation'' should be ``and from my observation''. 
(iv) Page 755, lines 3--4: ``manuscript (2013).'' should be ``manuscript, submitted to a journal on March 5, 2013 and rejected on May 9, 2013.''. 
\end {corrections}
\begin{flushleft}
Nagoya, Japan\\
\textit{E-mail address}: masahiro.igarashi2018@gmail.com
\end{flushleft}

\begin{thebibliography}{99}
\bibitem{an}
  G. E. Andrews, Problems and prospects for basic hypergeometric functions, 
in \textit{Theory and application of special functions}, Ed. by R. A. Askey, Math. Res. Center, Univ. Wisconsin, Publ. No. 35, Academic Press, 1975, New York, pp. 191--224.
\bibitem{ao}
  T. Aoki and Y. Ohno, Sum relations for multiple zeta values and connection formulas 
for the Gauss hypergeometric functions, Publ. Res. Inst. Math. Sci. \textbf{41} (2005), no. 2, 329--337.
\bibitem{ak}
  T. Arakawa and M. Kaneko, \textit{Introduction to multiple zeta values}, 
Kyushu University MI Lecture Note Series, Vol. 23 (2010) (in Japanese).
\bibitem{bbv}
  J. Bl\"{u}mlein, D. J. Broadhurst and J. A. M. Vermaseren, The multiple zeta value data mine, 
Comput. Phys. Comm. \textbf{181} (2010), no. 3, 582--625.
\bibitem{bbg}
  D. Borwein, J. M. Borwein and R. Girgensohn, Explicit evaluation of Euler sums, 
Proc. Edinburgh Math. Soc. (2) \textbf{38} (1995), no. 2, 277--294.
\bibitem{bbb}
  J. M. Borwein, D. M. Bradley and D. J. Broadhurst, Evaluations of $k$-fold Euler/Zagier sums: 
a compendium of results for arbitrary $k$, The Wilf Festschrift (Philadelphia, PA, 1996), 
Electron. J. Combin. \textbf{4} (1997), no. 2, Research Paper 5, approx. 21 pp. (electronic).
\bibitem{bbbl}
  J. M. Borwein, D. M. Bradley, D. J. Broadhurst and P. Lison\v{e}k, 
Special values of multiple polylogarithms, Trans. Amer. Math. Soc. \textbf{353} (2001), no. 3, 907--941.
\bibitem{bb}
  D. Bowman and D. M. Bradley, Multiple polylogarithms: a brief survey, 
in \textit{$q$-series with applications to combinatorics, number theory, and physics} (Urbana, IL, 2000), 
71--92, Contemp. Math., \textbf{291}, Amer. Math. Soc., Providence, RI, 2001.
\bibitem{eu}
  L. Euler, Meditationes circa singulare serierum genus, Novi Comm. Acad. Sci. Petropol. \textbf{20} 
(1775), 140--186; reprinted in Opera Omnia, Ser. I, Vol. 15, B. G. Teubner, Berlin, 1927, pp. 217--267.
\bibitem{h}
  M. E. Hoffman, Multiple harmonic series, Pacific J. Math. \textbf{152} (1992), no. 2, 275--290.
\bibitem{i1}
  M. Igarashi, Cyclic sum of certain parametrized multiple series, 
J. Number Theory \textbf{131} (2011), no. 3, 508--518.
\bibitem{i2011}
  M. Igarashi, Note on relations among multiple zeta-star values, 2011. arXiv:1106.0481. 
\bibitem{i4}
  M. Igarashi, Note on relations among multiple zeta(-star) values, Italian J. Pure Appl. Math. no. 39 (2018), 710--756; submitted on July 9, 2016.
\bibitem{ikoo}
  K. Ihara, J. Kajikawa, Y. Ohno and J. Okuda, Multiple zeta values vs. multiple zeta-star values, 
J. Algebra \textbf{332} (2011), no. 1, 187--208.
\bibitem{kr}
  C. Krattenthaler and T. Rivoal, An identity of Andrews, multiple integrals, and very-well-poised 
hypergeometric series, Ramanujan J. \textbf{13} (2007), no. 1-3, 203--219.
\bibitem{mu}
  S. Muneta, Algebraic setup of non-strict multiple zeta values, 
Acta Arith. \textbf{136} (2009), no. 1, 7--18.
\bibitem{n}
  T. Nakamura, Restricted and weighted sum formulas for double zeta values of even weight, \v{S}iauliai Math. Semin. \textbf{4} (12) (2009), 151--155.
\bibitem{ohza2001}
  Y. Ohno and D. Zagier, Multiple zeta values of fixed weight, depth and height, 
Indag. Math. (N. S.) \textbf{12} (2001), no. 4, 483--487.
\bibitem{whiwat}
  E. T. Whittaker and G. N. Watson, \textit{A Course of Mordern Analysis}, 4th ed., Cambridge Univ. Press, Cambridge, UK, 1927; reprinted 2003.
\bibitem{z}
  D. Zagier, Values of zeta functions and their applications, in \textit{First European Congress of Mathematics}, Vol. II (Paris, 1992), Ed. by A. Joseph et al., Progr. Math. \textbf{120}, Birkh\"{a}user, 1994, Basel, pp. 497--512.
\bibitem{zl2}
  S. A. Zlobin, Generating functions for the values of a multiple zeta function, 
Vestnik Moskov. Univ. Ser. I Mat. Mekh. \textbf{60} (2005), no. 2, 55--59 (in Russian); English transl., 
Moskow Univ. Math. Bull. \textbf{60} (2005), no. 2, 44--48.
\bibitem{zu}
  W. Zudilin, Well-poised hypergeometric service for diophantine problems of zeta values, 
J. Th\'{e}or. Nombres Bordeaux \textbf{15} (2003), no. 2, 593--626.
\bibitem{c}
  M.-A. Coppo, Nouvelles expressions des formules de Hasse et de Hermite pour la fonction 
Z\^{e}ta d'Hurwitz, Expo. Math. \textbf{27} (2009), no. 1, 79--86.
\bibitem{cc}
  M.-A. Coppo and B. Candelpergher, The Arakawa--Kaneko zeta function, Ramanujan J. \textbf{22} (2010), 153--162.
\bibitem{e}
  M. \'{E}mery, On a multiple harmonic power series, 2004. arXiv:math.NT/0411267.
\bibitem{fr}
  S. Fischler and T. Rivoal, Multiple zeta values, Pad\'{e} approximation and Vasilyev's conjecture, 
Ann. Sc. Norm. Super. Pisa Cl. Sci. (5) \textbf{15} (2016), 1--24.
\bibitem{h}
  H. Hasse, Ein Summierungsverfahren f\"{u}r die Riemannsche $\zeta$-Reihe, 
Math. Z. \textbf{32} (1930), no. 1, 458--464.
\bibitem{ig2007}
  M. Igarashi, On generalizations of the sum formula for multiple zeta values, Master's thesis, Graduate School of Mathematics, 
Nagoya University, Japan, February 3, 2007 (in Japanese). See also arXiv:1110.4875 and arXiv:0908.2536v5.
\bibitem{i2009}
  M. Igarashi, A generalization of Ohno's relation for multiple zeta values, J. Number Theory \textbf{132} (2012), 565--578.
\bibitem{i2020} 
  M. Igarashi, Duality relations among multiple series with three parameters, 
Tunisian J. Math. \textbf{2} (2020), no. 1, 217--236.
\bibitem{kaw2009}
  G. Kawashima, A class of relations among multiple zeta values, J. Number Theory \textbf{129} (2009), no. 4, 755--788. 
\bibitem{kaw}
  G. Kawashima, Multiple series expressions for the Newton series which interpolate finite multiple harmonic sums, 2009. arXiv:0905.0243v1.
\bibitem{ul}
  E. A. Ulanskii, Multiple zeta values, Vestnik Moskov. Univ. Ser. I Mat. Mekh. \textbf{66} (2011), no. 3, 14--19 (in Russian); 
English transl.: Moscow Univ. Math. Bull. \textbf{66} (2011), no. 3, 105--109.
\bibitem{i2022}
  M. Igarashi, On the duality formula for parametrized multiple series, 2022. 
http://arxiv.org/abs/2201.01651. 
\end{thebibliography}
\end{document}